\documentclass[12pt,letterpaper]{amsart}
\usepackage{array}
\usepackage{amsmath}
\newtheorem{theorem}{Theorem}
\newtheorem{lemma}{Lemma}

\begin{document}

\title{On the sum of a prime and a Fibonacci number}

\author{K. S. Enoch Lee}
\address{P O Box 244023, Department of Mathematics, Auburn University Montgomery, Montgomery,
AL 36124-4023, USA}
\email{elee4@aum.edu}




\begin{abstract}
We show that the set of the numbers that are the sum of a prime and a 
Fibonacci number has positive lower asymptotic density.
\end{abstract}

\maketitle

\section{Introduction}
Suppose $\mathcal S$ is a set of positive integers.  We denote the number of positive
 integers in $S$ not exceeding $N$ by $\mathcal S(N)$.  This function is called the {\it
counting function} of the set $\mathcal S$. The sumset, $\mathcal S + \mathcal T$, is the 
collection of the numbers of the form $s+t$ where $s\in \mathcal S$ and $t\in \mathcal T$. 

Suppose $\mathcal A=\{p+2^i :  p \text{ a prime} , i\geq 1\}$.  In 1934, 
Romanoff~\cite{Roma34} published the following interesting result.  For $N$ sufficiently
large, we have $\mathcal A(N) \geq cN$ for some $c>0$.
  In other words, the set $\mathcal A$ has a positive lower asymptotic
density.  Romanoff showed
that a positive proportion of positive integers can be decomposed into the form $p+2^i$.

Let $u_1=1, u_2=1, u_{i+2}=u_{i+1}+u_i$ where $i$ is a positive integer.
Denote by $\mathcal U$ the collection of Fibonacci numbers, namely 
$\mathcal U=\{u_i\}_{i\geq 2}$. Furthermore, let $\mathcal P$ denote the set of primes.
For convenience, we  stipulate that $p$ and $p'$ (with or without subscripts) are primes, 
and $u$ and $u'$ (with or without subscripts) are 
Fibonacci numbers.  Throughout this paper, we use the Vinogradov 
symbol $\ll$ and the Landau symbol $O$ with their usual meanings.

In this manuscript, we study the set of integers that are the sum of a prime and
a bounded number of Fibonacci numbers.  In view of
Romanoff's theorem,  a key element in the proof is
 $$\sum_{\substack{d=1
\\(2,d)=1}}^\infty \frac{\mu ^2(d)}{de(d)}\ll 1$$ where $e(d)$ is the exponent of $2$ 
modulo $d$.  
  
By using an estimate (\cite{Schi90}, \cite{Some90}) of the number of times the residue 
$t$ appeared in a full period
of $u_i \pmod p$, we are able to substitute the period $k(d)$ for $e(d)$ and prove that 
$\mathcal P +\mathcal U$ has a positive lower asymptotic density.  

\begin{theorem}\label{thm1}Suppose  
$$\mathcal F=\mathcal {P} + \mathcal{U}=\{p + u : 
p\in \mathcal P, u\in \mathcal U\}.$$  Then there is a positive 
constant $c$ such that $$\mathcal F(N)\geq cN$$ 
for all sufficiently large $N$.  
\end{theorem}  

As a consequence, the set $\mathcal {P} + k\mathcal{U}$ has a positive lower asymptotic density
for each $k\geq 1$, since $1\in \mathcal{U}$.
\section{Proof of the Theorem}
For our convenience, we let $L=[\log _\tau N]$ for a given $N$ and use this throughout this paper.
Let $\tau = (1+\sqrt{5})/{2}$.  It 
is well-known that $$\bigg |u_i - \frac{\tau ^i}{\sqrt{5}}\bigg |< \frac{1}{2}$$ for all $i\geq 1$.
Thus $u_i = \tau ^i/\sqrt{5} + O(1)$.  A routine computation yields that $$\mathcal
U(N)=L + O(1).$$

Denote by $r'(N)=\sum_{p+u=N} 1$ the 
number of solutions of the equation $N=p+u$ for $N\geq 1$.  We begin with the following lemma.

\begin{lemma}\label{lemr_k} For $N$ a large number, we have
$$\sum_{n\leq N} r'(n) \sim \frac{NL}{\log N}.$$
\end{lemma}

\begin{proof}
Note that 
$$
\pi (N-\frac{N}{L}) \mathcal{U}(\frac{N}{L})\leq
\sum_{n\leq N} r'(n) \leq \pi(N)\mathcal{U}(N).
$$
The lemma then follows from the prime number theorem.
\end{proof}
Properties of Fibonacci numbers can be found in standard texts such as 
\cite{Kosh01}, \cite{Vajd89}, and \cite{Voro02}.  For our discussions,
we recall some properties of Fibonacci numbers without providing proofs.  Given a 
positive integer $n$, there is a unique decomposition of $n$ into the sum of non-consecutive
 Fibonacci numbers, namely,
$$
n=u_{i_1} + u_{i_2}+\cdots +u_{i_r}
$$ 
where $2\leq i_r$ and $2\leq i_{j} -i_{j+1}$.  This is called Zeckendorf representation 
\cite{Brow64} (or canonical representation).  In other words, if $2\leq i_r$ and 
$2\leq i_{j} -i_{j+1}$, the set of integers $(i_1,i_2, \ldots, i_r)$ is uniquely determined
by $n$ and conversely.   It is well-known that $u_i \pmod d$ forms a purely periodic series
 \cite{Wall60}.  Let $k(d)$ denote the period of  Fibonacci numbers
modulo $d$.  That is to say $k(d)$ is the smallest positive integers $m$ such that
$u_{i+m} \equiv  u_i \pmod d$ for all $i$.  In particular, $d|u_{k(d)}$.  
Furthermore, the period $k(d)$ is equal to
 the least common multiple of $\{k(p_1^{r_1}), k(p_2^{r_2}), \ldots, k(p_t^{r_t})\}$ where
 $d=p_1^{r_1}\cdots  p_t^{r_t}$.   We also have that $k(d) | k(m)$ if $d\, | m$.
 
Let us investigate the following example.    The table below presents one period of the
residues for $u_i \pmod 6$, $u_i \pmod 2$, and $u_i \pmod 3$, respectively, where $i\geq 2$.

\begin{center}
\Small\addtolength{\tabcolsep}{-2.2pt}
\begin{tabular}{l|llllllllllllllllllllllll}
\!$i$ & 2 & 3 & 4 & 5 & 6 & 7 & 8 & 9 & 10 & 11 & 12 & 13 & 14& 15 & 
16 & 17 & 18 & 19 & 20 & 21 & 22 & 23 & 24 & 25 \\ \hline
\!$u_i\!\!\! \pmod 6$\!\! & 1 & 2 & 3 & 5 & 2 & 1 & 3 & 4 & 1 & 5 & 0 & 5 & 5 & 4 & 3 & 1 
& 4 & 5 & 3 & 2 & 5 & 1 & 0& 1 \\ 
\!$u_i\!\!\! \pmod 2$\!\! & 1 & 0 & 1 &  &  &  &  &  &  &  &  &  &  &  &  &  
&  &  &  &  &  &  &  &\\ 
\!$u_i\!\!\!  \pmod 3$\!\! & 1 & 2 & 0 & 2 & 2 & 1 & 0 & 1 &  &  &  &  &  &  &  &  
&  &  &  &  &  &  & &
\end{tabular}
\end{center}

From the table, we see that $k(6)=24, k(2)=3, \text{ and } k(3)=8$.  We note that 
$k(6)=LCM[k(2),k(3)]$.    For any modulus $d\geq 2$, 
and residue $y \pmod d$, denote by $\nu(d, y)$ the number of occurrences of $y $ as a residue in one 
full period of $u_i \pmod d$.  
Let us explore the case $y\equiv 5 \pmod d$.  From the table, we have $\nu (6,5)=6$, 
$\nu(2,5)=\nu(2,1)=2$, and $\nu (3,5)=\nu(3,2)=3$.  It is clear that $\nu (d,y) \leq \dfrac{k(d)}{k(p)}
\nu (p,y)$ where $p|d$.  We now return to our proof.

\begin{lemma}\label{lemf_kk}
Let $-N\leq h\leq N$ and $f(h)$ be the number of solutions of the equation:
$$
u - u' = h
$$
where $u, u'\leq N$.  Then
\begin{enumerate}
 \item $f(0) \sim L$ and $f(h)\leq 2$ if $h\not = 0$;
 \item  Suppose $d>1$ is an integer and $p|d$.  Then 
$$\sum_{d | h} f(h)\leq 4L\bigg (1+ \frac{L}{k(p)}\bigg ).$$
\end{enumerate}
\end{lemma}
\begin{proof}
Without loss of generality, we can assume $h\geq 0$.  (1) 
Clearly we have $f(0)={\mathcal U}(N)\sim L$.  Next
we claim that $f(h)\leq 2$ when $h>0$.  
Assume that $h=u_j-u_i=u_t-u_s$ where $j>i$ and $t>s$.  If $j-i=1$, then 
$h=u_j-u_{j-1}=u_{j-2}=u_{j-1}-u_{j-3}$.  Suppose $t>j$.
If $t-s=1$, we have $u_{t-2}=u_{j-2}$, a 
contradiction. If 
$t-s>1$, we have $u_t-u_s>u_t-u_{t-1}=u_{t-2}$, a contradiction again!  Suppose now $t<j-1$.  
This forces $u_t-u_s<u_{j-2}=h$.  Therefore, there are only two decompositions of $h$ into
the difference of two Fibonacci numbers, namely $u_j-u_{j-1}$ and $u_{j-1}-u_{j-3}$.

Now suppose $j-i\geq 2$.    From the definition of Fibonacci
numbers, we derive that
\begin{equation*}
u_j-u_i = 
	\begin{cases}
		u_{j-1} +u_{j-3} + \cdots +u_{j-(2v-1)},& \text{ if } i=j-2v;\\
		u_{j-1} +u_{j-3}+\cdots + u_{j-(2v-1)} + u_{j-(2v+2)},& \text{ if }i=j-(2v+1);
	\end{cases}
\end{equation*}
where $v\geq 1$.  Clearly these are Zeckendorf representations.
By the same token, $u_t-u_s$ has similar decompositions. The uniqueness of Zeckendorf 
representation implies that $t=j$ and thus $s=i$.  As a consequence, 
$f(h) \leq 2$.  

(2) The sum $\sum_{d|h} f(h)$ is the number of solutions of the congruence $u\equiv 
u' \pmod d$.   Note that $u\equiv u' \pmod p$ if $p|d$.  
However, Schinzel \cite{Schi90} and Somer \cite{Some90} showed that 
$\nu (p,y)\leq 4$, namely, there are at most 4 choices for $u$ in any interval of length $k(p)$ such that $u\equiv y \pmod p$.  This implies within an interval of length $L$ there are at most 
$4(1+ \frac{L}{k(p)})$ solutions to $u\equiv y \pmod d$.  Thus $p|d$ implies
$$\sum_{d | h} f(h)\leq 4L\bigg (1+ \frac{L}{k(p)}\bigg ).$$
\end{proof}

 \begin{lemma}\label{lemr_k2}For $k\geq 1$ and $N$ sufficiently large, we have
 $$
\sum_{n\leq N} (r'(n))^2 \leq c \frac{NL^{2}}{(\log N)^2} 
 $$
where $c>0$. 
 \end{lemma}
 \begin{proof}
In the following, we assume that $p, p' , u, u'\leq N$. We first break the sum into 
three parts.  
 \begin{eqnarray*}
 \sum_{n\leq N} (r'(n))^2 & = & \sum_{n\leq N} \bigg ( \sum_{p+u=n} 1\bigg )^2\\
 & = & \sum_{n\leq N}\sum_{\substack{p+u=n\\p'+u'=n}} 1\\
 & = &  \sum_{-N \leq h\leq N}\bigg(\sum_{\substack{p-p'=h\\p,p'\leq N}}1 \bigg)
 f(h).
 \end{eqnarray*}
Let $$\sum {(h)}= \bigg(\sum_{\substack{p-p'=h\\p,p'\leq N}}1  \bigg)
f(h),$$
where $h=0$, $h>0$, $h<0$.  We investigate these three cases respectively.
First, suppose $h=0$. From Lemma~\ref{lemf_kk}, we have 
$$
\sum (0)= \bigg( \sum_{p\leq N} 1\bigg)f(0) \sim
\frac{NL}{\log N}.
$$
Next, we suppose $h>0$ and is odd.  This implies $p'=2$, since $p-p'=h$. 
Thus
$$
\sum_{\substack{0<h\leq N\\2\not \,\,| h}}\sum(h) = 
 \sum_{\substack{0< h\leq N\\ 2\not \,\,| h}}\bigg(\sum_{\substack{p=h+2\\p\leq N}}1 \bigg)
f(h).
$$ 
Therefore, we have 
$$
\sum_{\substack{0<h\leq N\\2\not \,\,| h}}\sum(h) \ll
 \sum_{\substack{0< h\leq N\\ 2\not \,\,| h}}\bigg(\sum_{\substack{p=h+2\\p\leq N}}1 \bigg)
\ll  \frac{N}{\log N}.
$$ 
We now assume $h>0$ is even.  Recall that the number of primes
$p\leq N$ such that $p+h$ is also a prime is given by (cf.  
\cite[p.102]{Cojo06}, \cite[p.97]{Mont07} , and \cite[p.190]{Nath96})
$$
O\bigg(\frac{N}{(\log N)^2} \prod_{p|h} \bigg (1+ \frac{1}{p}\bigg)\bigg).
$$ 
By using Lemma~\ref{lemf_kk}, we obtain that 
\begin{eqnarray*}
\sum_{\substack{0<h\leq N\\2| h}}\sum(h) &\ll &
 \frac{N}{(\log N)^2}\sum_{0< h\leq N}f(h)\prod_{p|h}\bigg(1+ \frac{1}{p}\bigg)\\
&\ll& \frac{N}{(\log N)^2}\sum_{d\leq N}\frac{\mu^2(d)}{d}\,
\sum_{\substack{0<h\leq N\\d|h }} f(h)\\
&\ll& \frac{NL^{2}}{(\log N)^2} + \frac{NL}{(\log N)^2}
\sum_{1<d\leq N}\frac{\mu^2(d)}{d}\bigg(1+\frac{L}{k(p)}\bigg),
\end{eqnarray*}
where $p$ is a prime factor of $d$.   
 For our investigation, we let the function 
$LP(d)=\max \{p|d : k(p)\geq k(p') \text{ for } p'|d\}.$
We are to show  that 
$$\sum_{\substack{d\leq N\\p=LP(d)}}\frac{\mu^2(d)}{dk(p)}\ll 1.
$$ 

We define
$$
E(x)=\sum_{g\leq x}\sum_{\substack{p=LP(d)\\k(p)=g}}\frac{\mu^2(d)}{d}.  \eqno (*)
$$
In 1974, Catlin \cite{Catl74} showed that if $k(m)<2t$ then $m<L_t$ where $L_t$ is the $t$-th 
Lucas number.  Therefore, for a fixed number $g$, there are only finitely many solutions 
$p$ to the equation $k(p)=g$.   Furthermore, there can only be a finite number of primes having
period less than or equal to $k(p)$, and thus there are only finitely many squarefree $d$ having 
$p=LP(d)$.  This means $E(x)$ is well-defined.  Let $$D(x)=\prod _{i\leq x} u_i.$$
Without loss of generality, we assume that $d$, appearing in  the sum (*), is squarefree. 
Note that $p'|d$ implies $p'|u_{k(p')}|D(x)$.  We then have $d|D(x)$ since
$k(p)\leq x$ and $p=LP(d)$.   It is also clear that the number $d$ appears in (*) once.  
Let $n=\omega (D(x))$ be the number of distinct prime factors of $D(x)$.  Then
$$
2^n\leq D(x) \ll \prod_{i\leq x}\tau ^{i}\ll \tau^{x^2}.
$$
In other words, we have $n\ll x^2$, and thus $\log p_n \ll \log n\ll x$ (where $p_i$ is 
the $i$-th prime).  Immediately, we have 
$$
E(x)\ll \sum_{d|D(x)}\frac{\mu^2(d)}{d}=\prod_{p|D(x)}\bigg( 1+ \frac{1}{p}\bigg)\ll
\prod_{i=1}^n\bigg(1 + \frac{1}{p_i}\bigg).
$$
Apply Merten's formula to the last term to obtain
$$
E(x)\ll \log p_n \ll \log x.
$$
By partial summation, we have
$$
\sum_{g\leq x} \frac{1}{g}\, \sum_{\substack{p=LP(d)\\k(p)=g}}\frac{\mu^2(d)}{d} =
\frac{E(x)}{x} + \int_1^x \frac{E(x)}{t^2} dt\ll 1.
$$
This implies 
$$
\lim_{x\to\infty} \sum_{\substack{d\leq x\\p=LP(d)}}\frac{\mu^2(d)}{dk(p)}  = \lim_{x\to\infty}
\sum_{g\leq x}\frac{1}{g}\sum_{\substack{p=LP(d)\\k(p)=g}}\frac{\mu^2(d)}{d}\ll 1.
$$
As a consequence, we have
$$\sum_{0<h\leq N}\sum(h) \ll  
\frac{NL^{2}}{(\log N)^2}.
$$
By symmetry, 
$$
\sum_{-N\leq h<0}\sum(h) \ll  
\frac{NL^{2}}{(\log N)^2}.
$$
Combining the above estimations, we obtain
$$
\sum_{n\leq N}(r'(n))\ll \frac{NL^{2}}{(\log N)^2}.
$$
\end{proof}

Invoking the Cauchy-Schwarz inequality, we have
$$
\bigg(\sum_{n\leq N} r'(n)\bigg)^2\leq \mathcal F(N)\sum_{n\leq N}(r'(n))^2.
$$
However, Lemma~\ref{lemr_k} and Lemma~\ref{lemr_k2} imply  
$$
\mathcal F(N) \geq \frac{\bigg(\sum_{n\leq N} r'(n)\bigg)^2}{\sum_{n\leq N}(r'(n))^2}\geq  \frac{1}{c }N.
$$
This proves the theorem.  
\section{Remarks}
To conclude our paper, we post the following questions related to our quest.
\begin{enumerate}
\item  Is $r'(n)\ll 1$?  The referee notices that for any fixed $k\geq 2$, we can choose distinct Fibonacci
numbers $u_{m_1} , u_{m_2}, \cdots, u_{m_k}$ such that for any prime $p$ there exists $1\leq d_p\leq p$ 
satisfying $u_{m_i} \not \equiv d_p \pmod p$ for each $1\leq i\leq k$ (see Schinzel \cite[Corollary 1]{Schi90}).  
Then by the widely believed prime k-tuple conjecture (see \cite{Mont07}), there exist infinitely many $n$ such that $n-u_{m_1},
n-u_{m_2},\cdots, n-u_{m_k}$ are all primes.  That is, $r'(n)\geq k$.  
Thus the referee suggests that $\limsup_{n\to \infty} r'(n)=+\infty$ instead.
\item Find an infinite sequence (or an arithmetic progression) of positive integers
 that each of the terms cannot be of the form $p+u$.
  Note  Wu and Sun \cite{WuSu09} constructed a class
that does not contain integers representable as the sum of a prime and half of a Fibonacci number.
\item Is there a positive integer $k$ such that  $n$ can be decomposed into a sum of a 
prime and $k$ Fibonacci numbers for $n$ sufficiently large?  Note that Sun \cite{Sun09} 
has recently conjectured that every integer ($> 4$) can be written as the sum of an odd
prime and two positive Fibonacci numbers.
\end{enumerate}

\end{document}